\documentclass[11pt]{amsart}

\usepackage[english]{babel} 
\usepackage[T1]{fontenc}
\usepackage[utf8]{inputenc}

\usepackage{indentfirst}
\usepackage[left=2.5cm, right=2.5cm, top=2cm, bottom=2cm]{geometry} 
\usepackage{amsmath,amssymb,mathrsfs} 
\usepackage{graphicx}  
\usepackage{float, caption} 
\usepackage{wrapfig} 
\usepackage{framed}
\usepackage{array, multirow, tabularx} 
\usepackage{tikz} 
\usepackage{pifont} 
\usetikzlibrary{matrix,arrows,decorations.pathmorphing,patterns}
\usepackage{mathtools}

\usepackage[bookmarksnumbered=true,pdfstartview=FitH,pdftex,pdfborder={0 0 0},linkcolor=blue,citecolor=blue,colorlinks=true]{hyperref} 

\usepackage{multido}

{\newtheorem{de}{Definition}[section]}

{\newtheorem{theo}[de]{\textsc{Theorem}}}

{\newtheorem{prop}[de]{Proposition}}

{\newtheorem{lem}[de]{Lemma}}

{\newtheorem{cor}[de]{Corollary}}

{\theoremstyle{remark}
\newtheorem{remar}[de]{Remark}}

{\theoremstyle{remark}
}

{\theoremstyle{remark}
}

{\theoremstyle{remark}
\newtheorem{nota}[de]{Notation}}

\newcommand{\C}{\ensuremath{\mathbb{C}}}

\newcommand{\co}{\ensuremath{\mathscr{O}}}
\newcommand{\wt}[1]{\ensuremath{\widetilde{#1}}}
\newcommand{\m}{\ensuremath{\mathfrak{m}}}

\newcommand{\mc}[1]{\ensuremath{\mathcal{#1}}}

\newcommand{\lrp}[1]{\ensuremath{\left( #1\right)}}
\newcommand{\lra}[1]{\left\{#1\right\}}

\newcommand{\Cxy}{\ensuremath{\C\{x,y\}}}

\newcommand{\val}{\ensuremath{\mathrm{val}}}

\newcommand{\Z}{\ensuremath{\mathbb{Z}}}

\newcommand{\undun}{\underline{1}}

\newcommand{\N}{\ensuremath{\mathbb{N}}}
\newcommand{\unp}{\left\{1,\ldots,p\right\}}

\title{{Symmetry of maximals for fractional ideals of curves}}
\author[D.~Pol]{Delphine Pol}
\thanks{Research supported by a Japan Society for the Promotion of Science (JSPS) Postdoctoral Fellowship (Short-term) for North American and European Researchers.}
\address{
Delphine Pol\\
Department of Mathematics, Hokkaido University\\
 Kita 10, Nishi 8, Kita-ku\\
 Sapporo 060-0810\\
 Japan
 }
 \email{\href{pol@math.sci.hokudai.ac.jp}{pol@math.sci.hokudai.ac.jp}}
\date{\today}
\subjclass{14H20 (Primary), 14H50, 20M12}
 
\keywords{values, singular curves, fractional ideal, duality}

\newcommand{\supe}{>}
\newcommand{\infe}{<}

\begin{document}

\begin{abstract}
The purpose of this paper is to extend the symmetry of maximals of the ring of a germ of reducible plane curve proved by Delgado to a relation between the relative maximals of a fractional ideal and the absolute maximals of its dual for any admissible ring. In particular, it includes the case of germs of reduced reducible curve of any codimension. We then apply this symmetry to characterize the elements in the set of values of a fractional ideal from some of its projections and the irreducible absolute maximals of the dual ideal.
\end{abstract}

\maketitle

\section{Introduction}

Let $C$ be the germ of an irreducible plane curve with reduced local ring $\co_C$. If $t\in (\C,0)\mapsto (x(t),y(t))\in (\C^2,0)$ is a parametrization of the curve $C$, we define the valuation of an element $f\in\C\lra{x,y}$ as the order of $t$ of the series $f(x(t),y(t))$. For a germ of plane curve with $p$  branches, we define the value of an element $f\in\Cxy$ as the $p$-uple of its valuation along each branch. The semigroup is then the set of the values of the non zero divisors of $\co_C$, and it is a subset of $\Z^p$. 

\smallskip

The semigroup of a plane curve determines the equisingularity class of the curve (see \cite{zariskicourbe} and \cite{waldithese}). The explicit computation of the semigroup is considered for example in \cite{zariskicourbe}, \cite{campillo-delgado-gusein-generators} and \cite{delgado2}. In the irreducible case, the generators of the semigroup are determined by the characteristic exponents of the curve. The approach suggested in \cite{delgado2} for reducible plane curves is an inductive procedure which allows us to determine the semigroup of a plane curve with $p$ branches provided that we know the semigroups of the curves with $p-1$ components obtained by removing one of the components of $C$, and the values of maximal contact.

\smallskip

One can also define the set of values of any fractional ideal $I\subseteq \mathrm{Frac}(\co_C)$, where $\mathrm{Frac}(\co_C)$ is the total ring of fractions of $\co_C$. For a fraction $\frac{a}{b}\in\mathrm{Frac}(\co_C)$, we set $\val\lrp{\frac{a}{b}}=\val(a)-\val(b)$.  The set of values $\val(I)$ of $I$ is then the set of values of the non zero divisors of $I$. One can notice that this set is an ideal over the semigroup $\val(\co_C)$: if $a\in\val(I)$  and $b\in\val(\co_C)$, then $a+b\in\val(I)$. 

\smallskip

In the irreducible case, a standard basis of $I$ can be determined by the algorithm \cite[Theorem 2.4]{hefez-standard} (see definition~\ref{de:base:standard} for the notion of standard basis). The set of valuations of $I$ can then be deduced from this standard basis. In particular, this algorithm is used in \cite{hefez} for the computation of the set of values of K\"ahler differentials, which is a key ingredient for the analytic classification of plane branches.  The set of values of the Jacobian ideal and the set of values of its dual, namely the module of logarithmic residues, are studied in \cite{polcras} and \cite{polvalues} (see definition~\ref{de:dual} for the notion of dual). We suggest in \cite[\S 4.3.3]{polvalues} an algorithm for the computation of the values of the module of logarithmic residues for curves with exactly two branches which uses the algorithm \cite[Theorem 2.4]{hefez-standard}. However, this procedure cannot be extended to curves with three or more branches. 

\medskip

The computation of the semigroup of a plane curve given in \cite{delgado2} is based on a symmetry between two particular kinds of elements of the semigroup, which are called \emph{relative maximals} and \emph{absolute maximals} (see \ref{a:de:maximals} for the definitions). This result is related to the symmetry property of the semigroup of a Gorenstein curve proved by Delgado in \cite{delgado}. We extend the latter symmetry to any fractional ideal of a plane curve and a Gorenstein curve in respectively \cite{polcras} and  \cite{polvalues}, and then it is extended to any fractional ideal of more general rings called \emph{admissible rings} in \cite{kts}.

\medskip

The purpose of this paper is to study the set of values of a fractional ideal $I$ of an admissible ring, and in particular the properties of the set of its maximals.

\begin{theo}
\label{intro:theo:sym-max}
Let $R$ be an admissible local ring and let $I\subseteq\mathrm{Frac}(R)$ be a fractional ideal. Let $\alpha\in\val(I)$ and $\beta\in\val(I^\vee)$. We assume that $\alpha+\beta=\gamma-\undun$, where $\gamma$ is the conductor of $R$ (see definition~\ref{de:conductor}).
Then $\beta$ is an absolute maximal of $I^\vee$ if and only if $\alpha$ is a relative maximal of $I$.
\end{theo}

We then use this symmetry to investigate the computation of the set of values of a fractional ideal.

\medskip

Let us describe the content of this paper. 

\smallskip

 In section~\ref{sec:notations}, we recall several properties of the set of values of a fractional ideal which will be used in the next sections.

\smallskip

Section~\ref{sec:symmetry-max} is devoted to the proof of the main theorem~\ref{intro:theo:sym-max}. The proof of this theorem uses the symmetry theorem of \cite{polvalues} and \cite{kts} (see theorem~\ref{symmetry-values}) and properties of the set of values of fractional ideals. In particular, our proof gives an alternative proof of \cite[Theorem 2.10]{delgado2} without induction on the number of branches.  The proof in the basic case of Gorenstein curves relies on \cite{polvalues}, and the proof in the general case, as presented here using \cite{kts}, is very similar.

\smallskip

In section~\ref{sec:computation}, we investigate the computation of the set of values of a fractional ideal using induction on the number of branches as it is done in \cite{delgado2} for the semigroup of the curve. The generation theorem \cite[Theorem 1.5]{delgado2} can be generalized to any fractional ideal, and is not specific to plane curves (see theorem~\ref{theo:gen}). This theorem gives a characterization of the set of values of $I$ from some projections of  $\val(I)$ and the relative maximals of $I$. Thanks to theorem~\ref{intro:theo:sym-max}, determining the set of the relative maximals of $I$ is equivalent to determining the set of the absolute maximals of $I^\vee$. In subsection~\ref{subsection:irred:abs:max}, we study the set of the  absolute maximals of an ideal in the case of germs of analytic curves.

\subsection*{Acknowledgments.} The author is grateful to Felix Delgado for pointing out this question, to Michel Granger for useful discussions and comments, and to Mathias Schulze for his suggestion to consider admissible rings. The author also wants to thank Laura Tozzo and Philipp Korell for helpful comments and suggestions.

\section{Notations and preliminary results}
\label{sec:notations}

We recall in this section definitions and  properties from \cite{delgado2}, \cite{delgado}, \cite{polvalues} and \cite{kts} which will be used in the rest of this paper.

\subsection{Setup}

Let $C$ be the germ of  a reduced complex analytic curve, with $p$ irreducible components $C_1,\ldots,C_p$. We denote by $\co_C$ the reduced ring of $C$. The ring $\co_{C_i}$ of the branch $C_i$ is a one-dimensional integral domain, so that its normalization $\co_{\wt{C}_i}$ is isomorphic to $\C\lra{t_i}$ (see for example \cite[Corollary 4.4.10]{dejong}). The total ring of fractions of $\co_C$ satisfies (see \cite{dejong} for example): 
$$\mathrm{Frac}(\co_C)=\mathrm{Frac}(\co_{\wt{C}})=\bigoplus_{i=1}^p \mathrm{Frac}(\C\lra{t_i}).$$

\begin{de}
\label{de:value}
Let $g\in \mathrm{Frac}(\co_C)$. We define the \emph{valuation} of $g$ along the branch $C_i$ as the order of $t_i$ of the image of $g$ by the map $\mathrm{Frac}(\co_C)\to \mathrm{Frac}(\C\lra{t_i})$.  We denote the valuation of $g$ along $C_i$ by $\val_i(g)\in\Z\cup \lra{\infty}$, with the convention $\val_i(0)=\infty$.

We then define the \emph{value} of $g$ by $\val(g):=\lrp{\val_1(g),\ldots,\val_p(g)}\in \lrp{\Z\cup\lra{\infty}}^p$.  
\end{de}

The previous definition can be extended to more general rings introduced in \cite{kts}, which are called \emph{admissible rings}. We recall here the definition, and we refer to \cite{kts} for more details. We will only consider the local case. Properties of a semilocal ring can be deduced from the local case thanks to \cite[Theorem 3.2.2]{kts}.

We denote by $|\cdot|$ the cardinality of a set.

\begin{de}
\label{de:admissible}
Let $(R,\m)$ be a one dimensional Noetherian local Cohen-Macaulay ring. The ring $R$ is called \emph{admissible} if:
\begin{itemize}
\item the completion $\widehat{R}$ of $R$ is reduced,
\item the integral closure $\wt{R}$ of $R$ in $\mathrm{Frac}(R)$ satisfies $\wt{R}/\mathfrak{n}=R/\mathfrak{n}\cap R$ for any maximal ideal $\mathfrak{n}$ of $\wt{R}$,
\item we have $|R/\m|\geqslant |\mathscr{V}|$, where $\mathscr{V}$ is the set of discrete valuation rings of $\mathrm{Frac}(R)$ over $R$.
\end{itemize}

A value map $\val : \mathrm{Frac}(R)\to (\Z\cup \lra{\infty})^{|\mathscr{V}|}$ can be defined using the set of discrete valuation rings (see \cite[Definition 3.1.2]{kts}).
\end{de}

In particular, the ring $\co_C$ of the germ of a reduced curve is admissible, and the value map is the one defined in definition~\ref{de:value}. 

\medskip

We fix $R$ a local admissible ring,  we set $p=|\mathscr{V}|$, and $\val : \mathrm{Frac}(R)\to (\Z\cup\lra{\infty})^p$ the value map.

\begin{de}
\label{fracideal}
Let $I\subset  \mathrm{Frac}(R)$ be an $R$-module. We call $I$ a \emph{fractional ideal} if there exists a non zero divisor $g\in R$ such that $gI\subseteq R$ and if $I$ contains a non zero divisor of $\mathrm{Frac}(R)$. 
 We set:
$$\val(I):=\lra{\val(g)\vert g\in I \text{ non zero divisor }}\subset \Z^p.$$
\end{de}

For $I,J$ ideals in $\mathrm{Frac}(R)$, we set $(I:J)=\lra{a\in R \vert aJ\subseteq I}$.

\begin{de}
\label{canonical}
Let $K\subset \mathrm{Frac}(R)$ be a fractional ideal. We say that $K$ is a \emph{canonical ideal} if for all fractional ideals $I\subseteq \mathrm{Frac}(R)$, we have $$(K:(K:I))=I.$$

The ring $R$ is called Gorenstein if $R$ is a canonical ideal. 
\end{de}

\begin{prop}[\protect{\cite[Corollary 5.1.7]{kts}}]
\label{prop:kts:k0}
There exists a unique canonical ideal $K^0$ up to multiplication by an invertible element of $\wt{R}$ such that $$R\subseteq K^0\subseteq \wt{R}.$$
\end{prop}

For a Gorenstein ring $R$, we have $K^0=R$.

\begin{de}
\label{de:dual}
Let $I\subset \mathrm{Frac}(R)$ be a fractional ideal. The \emph{dual} of $I$ is:
$$I^\vee:=(K^0:I).$$

In particular, from the definition of $K^0$, we have $(I^\vee)^\vee=I$. 
\end{de}

\begin{remar}
We also have $I^\vee\simeq \mathrm{Hom}_R\lrp{I,K^0}$ (see for example \cite[Proof of Lemma 1.5.14]{dejong}).
\end{remar}

\begin{de}
\label{de:conductor}
The \emph{conductor ideal} of $R$ is $\mc{C}_R=\wt{R}^\vee$. In particular, there exists $\gamma\in\N^p$ such that $\mc{C}_R=t^\gamma\wt{R}$ and $\val(\mc{C}_R)=\gamma+\N^p$. We call $\gamma$ the \emph{conductor} of the ring $R$.
\end{de}

\subsection{Properties of the set of values of fractional ideals}

 Let $I\subseteq \mathrm{Frac}(R)$ be a fractional ideal. From the definition of a fractional ideal, one can notice that there exists $\lambda\in\Z^p$ such that \begin{equation}
\label{eq:lambda}
\val(I)\subseteq \lambda+\N^p.
\end{equation}

By \cite[Proposition 3.1.9]{kts}, the set of values of any fractional ideal $I\subseteq \mathrm{Frac}(R)$ is a \emph{good semigroup ideal}, which means that we have the following properties.

\begin{lem}[\protect{\cite[Proposition 3.1.9 (b)]{kts}}]
\label{a:lem:lambda:mu}
Let $I\subseteq \mathrm{Frac}(R)$ be a fractional ideal. There exists $\nu\in\Z^p$ such that:
$$\nu+\N^p\subseteq\val(I).$$
\end{lem}
We consider the product order on $\Z^p$ defined by: $$(\alpha_1,\ldots,\alpha_p)\leqslant (\beta_1,\ldots,\beta_p) \iff \forall i\in\unp, \alpha_i\leqslant \beta_i.$$ In particular, for $\alpha,\beta\in\Z^p$, $\inf(\alpha,\beta)=(\min(\alpha_1,\beta_1),\ldots,\min(\alpha_p,\beta_p))$.

\begin{prop}[see \protect{\cite[Proposition 3.1.9 (c)]{kts}}]
\label{a:prop:inf}
Let $I$ be a fractional ideal and $\alpha,\beta\in\val(I)$. Then $\inf(\alpha,\beta)\in\val(I)$.
\end{prop}

\begin{prop}[see \protect{\cite[Proposition 3.1.9 (d)]{kts}}]
\label{a:prop:valquimonte}
Let $I$ be a fractional ideal and $\alpha,\beta\in\val(I)$. Let us assume that $\alpha\neq \beta$ and that there exists $i\in\unp$ such that $\alpha_i=\beta_i$. Then there exists $\eta\in\val(I)$ such that:
\begin{enumerate}
\item for all $j\in\unp, \eta_j\geqslant \min(\alpha_j,\beta_j)$,
\item $\eta_i>\alpha_i$,
\item for all $j\in\unp$ such that $\alpha_j\neq \beta_j$, $\eta_j=\min(\alpha_j,\beta_j)$. 
\end{enumerate} 
\end{prop}

\begin{remar}
Proposition~\ref{a:prop:valquimonte} will often be used in the following, and if $\alpha,\beta,i$ satisfy the assumptions of proposition~\ref{a:prop:valquimonte}, we will say that we apply proposition~\ref{a:prop:valquimonte} to the triple $(\alpha,\beta,i)$.
\end{remar}

The following lemma is a direct consequence of the definition of $I^\vee$. 

\begin{lem}
\label{a:lem:i:idual}
Let $\alpha\in\val(I)$ and $\beta\in\val(I^\vee)$. Then $\alpha+\beta\in\val(K^0)$.
\end{lem}

Let us recall the symmetry theorem that will be used in the proof of our main theorem~\ref{intro:theo:sym-max}. We first need the following notations.

\begin{nota}
\label{nota:delta}
Let $\alpha=(\alpha_1,\ldots,\alpha_p)\in\Z^p$ and $\mc{E}\subseteq \Z^p$ an arbitrary subset of $\Z^p$.
\begin{itemize}
\item Let $i\in\unp$. We set:
$$\Delta_i(\alpha,\mc{E})=\lra{v\in\mc{E}\vert v_i=\alpha_i \text{ and } \forall j\neq i, v_j\supe \alpha_j}.$$
We then define $\Delta(\alpha,\mc{E})=\bigcup_{i=1}^p \Delta_i(\alpha,\mc{E})$.
\item  For $J\subseteq \unp$ we set:
$$\Delta_J(\alpha,\mc{E})=\lra{v\in\mc{E}\vert \forall j\in J, v_j=\alpha_j \text{ and } \forall j\notin J, v_j>\alpha_j}.$$
\end{itemize}
For a fractional ideal $I\subseteq \mathrm{Frac}(R)$, we denote for all $J\subseteq\lra{1,\ldots,p}$,  $\Delta_J(\alpha,I)=\Delta_J(\alpha,\val(I))$ and $\Delta(\alpha,I)=\Delta(\alpha,\val(I))$.
\end{nota}

The following symmetry theorem is proved in \cite{polvalues} for Gorenstein curves, and is extended to admissible rings in \cite{kts}:
\begin{theo}[\protect{\cite[Theorem 1.2]{polvalues}, \cite[Theorem 5.3.4, Lemma 5.2.8]{kts}}]
 \label{symmetry-values}
 Let $R$ be an admissible ring with canonical ideal $K^0$ as in proposition~\ref{prop:kts:k0}. Let $I\subseteq \mathrm{Frac}(R)$ be a fractional ideal. Then, for all $v\in\Z^p$:
 \begin{equation}
 \label{symmetry-values-eq}
 v\in \val(I^\vee)\iff \Delta(\gamma-v-\undun,I)=\emptyset.
 \end{equation}
\end{theo}
The previous theorem generalizes \cite[Theorem 2.8]{delgado} which characterizes Gorenstein curves by the symmetry of the semigroup. 

\subsection{Absolute and relative maximals}

The following definitions and properties are generalizations to fractional ideals of the ones given in \cite{delgado2}. 

Let $I\subseteq \mathrm{Frac}(R)$ be a fractional ideal.

\begin{de}
\label{a:de:maximals}
Let $\alpha\in\val(I)$.
\begin{enumerate}
\item If $\Delta(\alpha,I)=\emptyset$, we call $\alpha$ a \emph{maximal} of $I$.
\item If for all $J\subseteq \unp, J\neq \unp$ and $J\neq \emptyset$ we have $\Delta_J(\alpha,I)=\emptyset$ then we call $\alpha$ an \emph{absolute maximal} of $I$.
\item If $\Delta(\alpha,I)=\emptyset$ and for all $J\subseteq\unp$ such that $|J|\geqslant 2$ we have $\Delta_J(\alpha,I)\neq \emptyset$ then we call $\alpha$ a \emph{relative maximal} of $I$.
\end{enumerate}
\end{de}

\begin{remar}
The three notions of maximals, absolute maximals and relative maximals coincide in the case $p=2$. If $p=1$, the set of maximals of any fractional ideal is empty. From now on, we assume that $p\geqslant 2$. 
\end{remar}

\begin{remar}
\label{remar:max:finite}
Let $\lambda, \nu\in\Z^p$ be such that $\nu+\N^p\subseteq \val(I)\subseteq \lambda+\N^p$. Let $\alpha$ be a maximal of $I$. It follows from the fact that $\alpha\in \val(I)$ that $\alpha\geqslant \lambda$, and since $\Delta(\alpha,I)=\emptyset$, we also have $\alpha\infe \nu$. Therefore, the set of the maximals of $I$ is contained in $\lra{v\in\Z^p \vert \lambda\leqslant v\infe \nu}$, so that it is a finite set.
\end{remar}

The following lemma is a generalization of \cite[Lemma 1.3]{delgado2} to any fractional ideal. The proof is essentially the same as for the ring $\co_C$ of a plane curve. 

\begin{lem}
\label{a:lem:relmax}
Let $\alpha\in\Z^p$ be such that there exists $i\in\unp$ satisfying: 
\begin{enumerate}
\item  $\Delta_i(\alpha,I)=\emptyset$,
\item for all $j\neq i$, $\Delta_{i,j}(\alpha,I)\neq \emptyset$.
\end{enumerate}
Then $\alpha$ is a relative maximal of $I$. 
\end{lem}
\begin{proof}
For all $j\neq i$, let $\alpha^j\in\Delta_{i,j}(\alpha,I)$. Then $\alpha=\inf(\lrp{\alpha^j}_{j\neq i})$ so that by proposition~\ref{a:prop:inf}, we have $\alpha\in\val(I)$. 

Let us assume that there exists $k\in\unp$ such that $\Delta_k(\alpha,I)\neq \emptyset$. Let $\eta\in\Delta_k(\alpha,I)$. Since $\eta_k=\alpha^k_k$ and $\eta\neq\alpha^k$, by proposition~\ref{a:prop:valquimonte} applied to the triple $(\eta,\alpha^k,k)$, there exists $\mu\in\val(I)$ such that $\mu_k>\alpha_k^k$, $\mu_i=\alpha_i^k=\alpha_i$ and for all $\ell\notin\lra{i,k}$, $\mu_\ell\geqslant \min(\alpha_\ell^k, \eta_\ell)$. Since $\alpha_\ell^k>\alpha_\ell$ and $\eta_\ell>\alpha_\ell$, we have $\mu\in\Delta_i(\alpha,I)$, which is impossible. Therefore, for all $k\in\unp$, $\Delta_k(\alpha,I)=\emptyset$. 

\smallskip

Let us prove that for all $J=\lra{k,\ell}\subseteq \unp\backslash\lra{i}$, $\Delta_J(\alpha,I)\neq \emptyset$. Since $\alpha_i^k=\alpha_i^\ell=\alpha_i$, and $\alpha^k\neq \alpha^\ell$, by proposition~\ref{a:prop:valquimonte} applied to the triple $(\alpha^k,\alpha^\ell,i)$, there exists $\eta\in\val(I)$ such that $\eta_i>\alpha_i$, $\eta_k=\alpha^k_k=\alpha_k$, $\eta_\ell=\alpha^\ell_\ell=\alpha_\ell$ and for all $j\notin\lra{i,k,\ell}, \eta_j\geqslant\min(\alpha_j^k,\alpha_j^\ell)>\alpha_j$. Therefore, $\eta\in\Delta_{k,\ell}(\alpha,I)$. 

\smallskip

Let us consider now $J\subseteq\unp$ with $|J|\geqslant 2$. We set $J=J_1\cup J_2\cup\dots \cup J_k$ with  for all $j\in\lra{1,\ldots,k}$, $|J_j|=2$. For all $j\in\lra{1,\ldots,k}$, let $\eta^j\in\Delta_{J_j}(\alpha,I)$. Then by proposition~\ref{a:prop:inf}, $\inf(\eta^1,\ldots,\eta^k)\in\Delta_J(\alpha,I)$. Hence the result. 
\end{proof}

The following lemma is a generalization of \cite[Lemma 2.8]{delgado2}.

\begin{lem}
\label{a:lem:2.8}
Let $\alpha\in\val(I)$. Let us assume\footnote{One can notice that $\alpha$ may not be a relative maximal because we do not assume that $\Delta(\alpha,I)=\emptyset$.} that for all $J\subseteq \unp$ such that $|J|\geqslant 2$, we have $\Delta_J(\alpha,I)\neq \emptyset$. Let $\beta=\gamma-\alpha-\undun$. Then for all $A\subseteq \unp$ with $A\neq \unp$ and $A\neq \emptyset$, we have $\Delta_A(\beta,I^\vee)=\emptyset$. 

If in addition $\beta\in\val(I^\vee)$, then $\beta$ is an absolute maximal of $I^\vee$ and $\alpha$ is a relative maximal of $I$. 
\end{lem}

\begin{proof}
We set $\beta=\gamma-\alpha-\undun$. 
Let us assume that there exists $\emptyset\neq A\subseteq\unp$, $A\neq \unp$, such that $\Delta_A(\beta,I^\vee)\neq \emptyset$. Let  $\eta\in\Delta_A(\beta,I^\vee)$. We set $J=A^c\cup\lra{i}$ with $A^c$ the complement of $A$ in $\unp$ and $i\in A$. Let $\mu\in\Delta_J(\alpha,I)$. Then $\eta+\mu\in\Delta_i(\gamma-\undun,K^0)$. However, by \cite[Lemma 5.2.2]{kts}, $\gamma$ is also the conductor of $K^0$, and by \cite[Lemma 4.1.10]{kts}, we have $\Delta(\gamma-\undun,K^0)=\emptyset$. Hence the result.

If we assume in addition that $\beta\in\val(I^\vee)$, then $\beta$ is an absolute maximal of $I^\vee$, and by theorem~\ref{symmetry-values}, $\Delta(\alpha,I)=\emptyset$ so that $\alpha$ is a relative maximal of $I$.
\end{proof}
As an immediate corollary we have:
\begin{cor}
\label{cor:impl1}
If $\alpha\in\val(I)$ is a relative maximal of $I$ and $\beta=\gamma-\alpha-\undun\in\val(I^\vee)$, then $\beta$ is an absolute maximal of $I^\vee$. 
\end{cor}

\section{Symmetry of the maximals}
\label{sec:symmetry-max}

The purpose of this section is to prove that the converse implication of corollary~\ref{cor:impl1} is satisfied, which will give theorem~\ref{intro:theo:sym-max}.

Theorem~\ref{intro:theo:sym-max} is a generalization to any admissible ring and any fractional ideal of the theorem of symmetry between relative and absolute maximals of the ring of a plane curve proved in \cite{delgado2}.  We recall that if $p=1$, then there is no maximal, and if $p=2$, the statement can be rephrased as "$\beta$ is a maximal of $I^\vee$ if and only if $\alpha$ is a maximal of $I$", which is a direct consequence of theorem~\ref{symmetry-values}. 

The proof given in \cite{delgado2} uses an induction on the number of branches. We suggest here a different proof which does not use an induction on the number of branches.

\medskip

From now on, we assume that $p\geqslant 3$.

\begin{proof}[Proof of theorem~\ref{intro:theo:sym-max}]\

Let us assume that:\begin{itemize}
\item $\alpha\in\val(I)$, $\beta\in\val(I^\vee)$, $\alpha+\beta=\gamma-\undun$,
\item $\beta$ is an absolute maximal of $I^\vee$,
\item $\alpha$ is not a relative maximal of $I$. 
\end{itemize}

\begin{lem}
\label{b:lem:E}
There exists $j\in\lra{2,\ldots,p}$ such that $$\mc{E}_{1,j}:=\lra{v\in\Z^p\vert v_1=\beta_1, v_j=\beta_j, \forall i\notin\lra{1,j}, v_i<\beta_i}$$
satisfies the following property:
\begin{equation}
\label{b:eq:E}
\forall v\in\mc{E}_{1,j},\  \Delta(v,I^\vee)\neq \emptyset.
\end{equation}
\end{lem}
\begin{proof}
Since $\beta\in\val(I^\vee)$, by theorem~\ref{symmetry-values}, we have $\Delta(\alpha,I)=\emptyset$. In particular, $\Delta_1(\alpha,I)=\emptyset$. Therefore, by lemma~\ref{a:lem:relmax}, since $\alpha$ is not a relative maximal of $I$, there exists $j\neq 1$ such that $\Delta_{1,j}(\alpha,I)=\emptyset$. Let $v\in\mc{E}_{1,j}$, and let $w=\gamma-v-\undun$. In particular, $w_1=\alpha_1$, $w_j=\alpha_j$ and for all $\ell\notin\lra{1,j}, w_\ell>\alpha_\ell$. Since $\Delta_{1,j}(\alpha,I)=\emptyset$, then $w\notin\val(I)$. Therefore, by theorem~\ref{symmetry-values}, we have $\Delta(v,I^\vee)\neq\emptyset$. 
\end{proof}

\begin{remar}
By renumbering the branches, we assume that the index $j$ satisfying lemma~\ref{b:lem:E} is $j=2$, and we set $\mc{E}=\mc{E}_{1,2}$.
\end{remar}

We will prove  thanks to the following lemmas the existence of a maximal element $\eta$ in $\mc{E}$ satisfying the property $\Delta_1(\eta,I^\vee)\neq\emptyset$ and $\Delta_2(\eta,I^\vee)\neq \emptyset$ (see lemma~\ref{b:lem:eta}).

\begin{lem}
\label{b:lem:12}
Let $v\in\mc{E}$. Then $\Delta_1(v,I^\vee)\neq \emptyset$ if and only if $\Delta_2(v,I^\vee)\neq \emptyset$. 
\end{lem}
\begin{proof}
Let $v\in\mc{E}$. 
Let us assume that $\Delta_1(v,I^\vee)\neq \emptyset$. There exists $w\in\val(I^\vee)$ such that $w_1=v_1=\beta_1$, $w_2>v_2=\beta_2$ and for all $i\geqslant 3$, $w_i>v_i$. Since $\beta\in\val(I^\vee)$, and $w_1=\beta_1$,  by the proposition~\ref{a:prop:valquimonte} applied to the triple $(\beta,w,1)$, there exists $w'\in\val(I^\vee)$ such that $w_1'>\beta_1=w_1$, $w_2'=\beta_2$, and for all $i\geqslant 3$, $w_i'\geqslant \min(w_i,\beta_i)$. Since $\beta_1=v_1$, $\beta_2=v_2$ and for all $i\geqslant 3$, $\min(w_i,\beta_i)>v_i$, we have $w'\in\Delta_2(v,I^\vee)$. 
\end{proof}

\begin{lem}
\label{b:lem:beta}
Let $\beta'=(\beta_1,\beta_2, \beta_3-1,\ldots,\beta_p-1)$. In particular, $\beta'$ is the maximal element of $\mc{E}$ for the product order. Then $\Delta_1(\beta',I^\vee)=\Delta_2(\beta',I^\vee)=\emptyset$. 
\end{lem}
\begin{proof}
Let us assume that there exists $v\in\Delta_1(\beta',I^\vee)$. Let $J=\lra{i\in\unp\vert v_i=\beta_i}$. Then $v\in\Delta_J(\beta,I^\vee)$. In addition, $J\neq \emptyset$ since $1\in J$, and $J\neq \unp$ since $v_2>\beta_2$. It contradicts the fact that $\beta$ is an absolute maximal. Hence $\Delta_1(\beta',I^\vee)=\emptyset$ and by lemma~\ref{b:lem:12}, $\Delta_2(\beta',I^\vee)=\emptyset$. 
\end{proof}

\begin{lem}
\label{b:lem:lambda}
Let $\lambda\in\Z^p$ be such that $\val(I^\vee)\subseteq \lambda+\N^p$. Let $\lambda'=(\beta_1,\beta_2,\lambda_3-1,\ldots,\lambda_p-1)\in\mc{E}$. Then $\Delta_1(\lambda',I^\vee)\neq \emptyset$ and $\Delta_2(\lambda',I^\vee)\neq \emptyset$. 
\end{lem}
\begin{proof}
Since $\lambda'\in \mc{E}$, we have $\Delta(\lambda',I^\vee)\neq \emptyset$ by lemma~\ref{b:lem:E}. Since $\val(I^\vee)\subseteq \lambda+\N^p$, we have for all $i\geqslant 3$, $\Delta_i(\lambda',I^\vee)=\emptyset$. Therefore, using lemma~\ref{b:lem:12}, we have $\Delta_1(\lambda',I^\vee)\neq \emptyset$ and $\Delta_2(\lambda',I^\vee)\neq \emptyset$. 
\end{proof}

The following lemma is a consequence of lemmas~\ref{b:lem:lambda} and~\ref{b:lem:12}, and of the existence of a maximal element for the product order in $\mc{E}$. 
\begin{lem}
\label{b:lem:eta}
There exists $\eta\in\mc{E}$ such that:
\begin{enumerate}
\item \label{cond:1} $\Delta_1(\eta,I^\vee)\neq \emptyset$ and $\Delta_2(\eta,I^\vee)\neq \emptyset$,
\item for all $\eta'\in\mc{E}$ such that $\eta'\geqslant \eta$ and such that there exists $i_0\geqslant 3$ with $\eta_{i_0}'>\eta_{i_0}$, we have $\Delta_1(\eta',I^\vee)=\Delta_2(\eta',I^\vee)=\emptyset$. 
\end{enumerate}
In other words, $\eta$ is a maximal element in the subset of $\mc{E}$ composed of the elements satisfying the first condition~\eqref{cond:1}.
\end{lem}

\begin{nota}
We fix an element $\eta\in\mc{E}$ satisfying lemma~\ref{b:lem:eta}. We set: 
\begin{itemize}
\item $K=\lra{i\in\lra{3,\ldots,p}\vert \eta_i\leqslant \beta_i-2}$,
\item $J_1=\lra{i\in\lra{3,\ldots,p}\vert \eta_i=\beta_i-1}$.
\end{itemize}
In particular, $K\cup J_1=\lra{3,\ldots,p}$, and $K\neq\emptyset$ by lemma~\ref{b:lem:beta}.
\end{nota}

\begin{lem}
\label{b:lem:mu}
Let $\mu\in\Delta_1(\eta,I^\vee)$. Then:
\begin{enumerate}
\item $\mu_1=\beta_1$, $\mu_2>\beta_2$ and for all $j\in J_1$, $\mu_j\geqslant \beta_j$,
\item for all $i\in K$, $\mu_i=\eta_i+1$.
\end{enumerate}
\end{lem}
\begin{proof}
The first property comes from the definitions of $\Delta_1(\eta,I^\vee)$ and $J_1$. We have for all $i\in K$, $\mu_i\geqslant\eta_i+1$. Let us assume that there exists $i_0\in K$ such that $\mu_{i_0}>\eta_{i_0}+1$. Let $\eta'\in\mc{E}$ be such that for all $i\neq i_0$, $\eta_i'=\eta_i$ and $\eta_{i_0}'=\eta_{i_0}+1\leqslant \beta_{i_0}-1$. Then $\mu\in\Delta_{1}(\eta',I^\vee)$, which is impossible from the definition of $\eta$. 
\end{proof}

The following proposition is the initialization of the induction of proposition~\ref{prop:ite}.

\begin{prop}
\label{b:prop:initialisation}
Let $v^1\in\mc{E}$ be such that for all $i\in K$, $v_i^1=\eta_i+1$ and for $j\in J_1$, $v_j^1=\beta_j-1$. We set $$J_2=\lra{j\in\unp,\Delta_j(v^1,I^\vee)\neq \emptyset}.$$
Then $J_2\subseteq J_1$. In addition, $J_2\neq \emptyset$. 
\end{prop}
\begin{proof}
It follows from the definition of $\eta$ that $\Delta_1(v^1,I^\vee)=\Delta_2(v^1,I^\vee)=\emptyset$. Let us assume that there exists $i_0\in K$ such that $\Delta_{i_0}(v^1,I^\vee)\neq \emptyset$. Let $w\in\Delta_{i_0}(v^1,I^\vee)$. Let $\mu\in\Delta_1(\eta,I^\vee)$ as in lemma~\ref{b:lem:mu}. Since $\mu_{i_0}=w_{i_0}=\eta_{i_0}+1$ and for all $i\in K\backslash\lra{i_0}$, $\mu_i=\eta_i+1<w_i$, by proposition~\ref{a:prop:valquimonte} applied to the triple $(\mu,w,i_0)$, there exists $w'\in\val(I^\vee)$ such that $w_1'=\mu_1=\beta_1$, $w_2'\geqslant\min(w_2,\mu_2)>\beta_2$, for all $i\in K\backslash i_0$, $w_i'=\mu_i=\eta_i+1$, $w_{i_0}'>\eta_{i_0}+1$ and for all $j\in J_1$, $w_j\geqslant \min(\mu_j,w_j)\geqslant \beta_j$. Let $\eta'\in\mc{E}$ be such that for $i\neq i_0$, $\eta_i'=\eta_i$ and $\eta_{i_0}'=\eta_{i_0}+1$. Then $w'\in\Delta_1(\eta',I^\vee)$, which leads to a contradiction with the definition of $\eta$. Therefore, $J_2\subseteq J_1$. In addition, since $v^1\in\mc{E}$, we have $J_2\neq \emptyset$ by lemma~\ref{b:lem:E}.
\end{proof}

\begin{prop}
\label{prop:ite}
 Let $q\geqslant 2$. Let us assume that there exist a sequence of sets $$J_q\subseteq J_{q-1}\subseteq
\dots\subseteq J_2\subseteq J_1$$ and a sequence of elements of $\mc{E}$ $$v^1,v^2,\ldots, v^{q-1}$$ such that, if $q\geqslant 3$, for all $\ell\in\lra{2,\ldots,q-1}$, we have: 
\begin{enumerate}
\item for all $i\notin J_\ell, v_i^\ell=v_i^{\ell-1}$,
\item for all $j\in J_\ell, v_j^\ell=v_j^{\ell-1}-1$,
\item $J_{\ell+1}=\lra{j\in\unp\vert \Delta_j(v^\ell,I^\vee)\neq\emptyset}$. 
\end{enumerate}
Let $v^q\in\mc{E}$ be such that for all $i\notin J_q, v_i^q=v_i^{q-1}$ and for all $j\in J_q, v_j^q=v_j^{q-1}-1$. Let $J_{q+1}=\lra{j\in\unp\vert \Delta_j(v^q,I^\vee)\neq \emptyset}$. 
Then $J_{q+1}\subseteq J_q$ and $J_{q+1}\neq \emptyset$.
\end{prop}
\begin{remar}
One can notice that for all $\ell\in\lra{1,\ldots,q-1}$, for all $m\in\lra{1,\ldots,\ell-1}$ and for all $j\in J_m\backslash J_{m+1}$, $v_j^\ell=\beta_j-m$ and for all $j\in J_\ell$, $v_j^\ell=\beta_j-\ell$.  
\end{remar}

\begin{proof}
Let us assume that there exists $i_0\in \unp\backslash J_q$  such that $\Delta_{i_0}(v^q,I^\vee)\neq\emptyset$. If $w\in\Delta_{i_0}(v^q,I^\vee)$, then:
\begin{itemize}
\item $w_{i_0}=v_{i_0}^q=v_{i_0}^{q-1}$,
\item for all $i\notin J_q\cup{i_0}$, $w_i>v_i^q=v_i^{q-1}$,
\item for all $i\in J_q$, $w_i\geqslant v_i^{q}+1=v_i^{q-1}$.
\end{itemize}

 Since $i_0\notin J_q$, from the definition of $J_q$ we have $\Delta_{i_0}(v^{q-1},I^\vee)=\emptyset$. Therefore, there exists $j_0\in J_q$ such that $w_{j_0}=v_{j_0}^{q-1}=v_{j_0}^q+1$.

We set $$s=\inf_{w\in\Delta_{i_0}(v^q,I^\vee)} \lrp{\mathrm{Card}\lrp{\lra{j\in J_q\vert w_j=v_j^{q-1}}}}.$$

We then have $s\geqslant 1$. 

 Let us choose an element $w\in\Delta_{i_0}(v^q,I^\vee)$ such that the cardinality of $\lra{j\in J_q\vert w_j=v_j^{q-1}}$ is $s$. 
 
 Let $j_0\in\lra{j\in J_q\vert w_j=v_j^{q-1}}$. Since $j_0\in J_q$, there exists $w'\in\Delta_{j_0}(v^{q-1},I^\vee)$. We thus have:
\begin{itemize}
\item $w_{j_0}'=w_{j_0}=v_{j_0}^{q-1}= v_{j_0}^q+1$,
\item $w_{i_0}'>w_{i_0}=v^{q-1}_{i_0}$,
\item for all $i\notin J_q\cup \lra{i_0}$, $w_i'\supe v^{q-1}_i=v^q_i$,
\item for all $j\in J_q\backslash \lra{j_0}$, $w_j'\supe v_j^{q-1}\supe v_j^q$.
\end{itemize}

By proposition~\ref{a:prop:valquimonte} applied to the triple $(w,w',j_0)$, there exists $u\in\val(I^\vee)$ such that for all $i\notin J_q\cup\lra{i_0}$, $u_i>v_i^{q-1}=v_i^q$, $u_{i_0}=v_{i_0}^{q-1}=v_{i_0}^q$, $u_{j_0}>v_{j_0}^{q-1}>v_{j_0}^q$ and for all $j\in J_q\backslash \lra{j_0}$, $u_j\geqslant\min(w_j,w_j')>v_j^q=v_j^{q-1}-1$. In addition, for all $j\in J_q\backslash\lra{j_0}$, if $w_j>v_j^{q-1}$, then $u_j>v_j^{q-1}$.  Therefore, $u\in\Delta_{i_0}(v^q,I^\vee)$ and $\mathrm{Card}\lrp{\lra{j\in J_q\vert u_j=v_j^{q-1}}}=s-1$, which contradicts the minimality of  $s$. Hence the result: $J_{q+1}\subseteq J_q$.

By lemma~\ref{b:lem:E}, for all $\ell\in\lra{1,\ldots,q}$, since $v^\ell\in\mc{E}$, we have $\Delta(v^\ell,I^\vee)\neq \emptyset$, thus $J_{\ell+1}\neq \emptyset$. 
\end{proof}

We can now finish the proof of theorem~\ref{intro:theo:sym-max}.

By proposition~\ref{prop:ite}, for all $\ell\in\N$, $v^\ell\in\mc{E}$ so that $\Delta(v^\ell,I^\vee)\neq \emptyset$ and we have $J_{\ell+1}\neq \emptyset$. There exists $q\in\N$ such that for all $j\in J_1$, $\beta_j-q<\lambda_j$ where $\lambda$ is defined in lemma~\ref{b:lem:lambda}. Since for all $j\in J_q$, $v_j^q=\beta_j-q<\lambda_j$, $\Delta_j(v^q,I^\vee)=\emptyset$, which is impossible, and which finishes the proof of the missing implication of theorem~\ref{intro:theo:sym-max}: if $\beta$ is an absolute maximal of $I$ then $\alpha$ is a relative maximal of~$I$. 
\end{proof}

\begin{remar}
Since we only used combinatorial properties of the set of values in the proof of  theorem~\ref{intro:theo:sym-max}, this theorem can also be stated in the context of good semigroup ideals as considered in \cite{kts}. 
\end{remar}

\section{Description of the set of values of a fractional ideal}
\label{sec:computation}

\subsection{The generation theorem}

The purpose of this subsection is to generalize the generation theorem \cite[Theorem 1.5]{delgado2}. If $J\subseteq \lra{1,\ldots,p}$, we denote by $pr_J : \Z^p\to \Z^{|J|}$ the surjective map defined by $pr_J(v_1,\ldots,v_p)=(v_j)_{j\in J}$.  

 It follows from remark~\ref{remar:max:finite} that the set of the relative maximals of any fractional ideal is finite. 

\begin{theo}
\label{theo:gen}
Let $R$ be a local admissible ring, and $I\subseteq \mathrm{Frac}(R)$ be a fractional ideal. 

Let $RM(I)=\lra{\alpha^1,\ldots,\alpha^q}$ be the set of the relative maximals of $I$. Let $v\in\Z^p$ be such that for all $J\subseteq \lra{1,\ldots,p}$ with $|J|=p-1$, we have $pr_J(v)\in pr_J(\val(I))$. Then $v\in\val(I)$ if and only if for all $i\in\lra{1,\ldots, q}$, $v\notin\Delta(\alpha^i,\Z^p)$. 
\end{theo}
\begin{proof}
We extend the proof of \cite[Theorem 1.5]{delgado2} as follows. The first implication is immediate, since for all $i\in\unp$, $\Delta(\alpha^i,I)=\emptyset$. 

Let us denote for $i\in\unp$, $J\subseteq\unp$ and $\alpha\in\Z^p$:
$$\Delta_J^i(\alpha,I)=\lra{w\in \val(I)\vert \forall j\in J, w_j=\alpha_j \text{ and } \forall r\leqslant i, w_r\geqslant \alpha_r \text{ and } \forall s>i, s\notin J, w_s>\alpha_s}.$$

Let $v$ be such that for all $J\subseteq \lra{1,\ldots,p}$ with $|J|=p-1$, we have $pr_J(v)\in pr_J(\val(I))$, and such that for all $i\in\lra{1,\ldots,q}$, $v\notin\Delta(\alpha^i,\Z^p)$. Let us assume that $v\notin \val(I)$. Then, there exists $j\in\unp$ such that $\Delta_j^p(v,I)=~\emptyset$. Indeed, if for all $j\in\unp$, $\emptyset \neq  \Delta_j^p(v,I)\ni w^j$, then $v=\inf(w^1,\ldots,w^p)\in\val(I)$, which contradicts our assumption. By renumbering the branches, we can assume that $\Delta_1^p(v,I)=\emptyset$.

Let $i\in\unp$ be the smallest integer such that there exists $v_{i+1}',\ldots,v_p'\in\Z$, $\eta^{i+1},\ldots,\eta^p\in\val(I)$ such that by denoting $v^{(i+1)'}=(v_1,\ldots,v_i,v_{i+1}',\ldots,v_p')$  we have:
\begin{itemize}
\item for all $k\in\lra{i+1,\ldots,p}$, $v_k'<v_k$,
\item for all $k\in\lra{i+1,\ldots,p}$, $\eta^k\in\Delta_{1,k}^i(v^{(i+1)'},I)$,
\item $\Delta_1^i(v^{(i+1)'},I)=\emptyset$. 
\end{itemize} 
 One can notice that this condition is always satisfied for $i=p$ since we then have $v=v^{(p+1)'}$ so that $\Delta_1^p(v^{(p+1)'},I)=\emptyset$. 
 
 Let us assume that $i>1$. We recall that there exists $\lambda\in\N^p$ such that $\val(I)\subseteq\lambda+\N^p$. We set $v^{*}=(v_1,\ldots,v_{i-1},\lambda_i,v_{i+1}',\ldots,v_p')$. 
 
 By assumption, we have $(v_1,\ldots,v_{i-1},v_{i+1},\ldots,v_p)\in pr_{\lra{1,\ldots,i-1,i+1,\ldots,p}}(I)$ so that $\Delta^i_1(v^{*},I)\neq \emptyset$.
 
 In addition, since $\Delta_1^i(v^{(i+1)'},I)=\emptyset$, if $(v_1,w_2,\ldots,w_p)\in \Delta_1^i(v^*,I)$ then $w_i<v_i$.  Let $v_i'$ be the maximal integer such that there exists $\eta^i=(v_1,w_2,\ldots,w_{i-1},v_i',w_{i+1},\ldots,w_p)\in\Delta_1^i(v^*,I)$. In particular, $v_i'<v_i$. We set $v^{(i)'}=(v_1,\ldots,v_{i-1},v_i',v_{i+1}',\ldots,v_p')$. Then $\eta^i\in\Delta_{1,i}^{i-1}(v^{(i)'},I)$, and for all $k\in\lra{i+1,\ldots,p}$, $\eta^k\in \Delta_{1,k}^{i-1}(v^{(i)'},I)$. In addition, since $v_i'$ is maximal, $\Delta_1^{i-1}(v^{(i)'},I)=\emptyset$, so that $i$ was not minimal. 

\smallskip

Therefore, $i=1$, which means that there exists $v_2',\ldots,v_p'\in\N$ such that for all $k\in\lra{2,\ldots,p}$, $v_k'<v_k$, there exits $\eta^k\in\Delta_{1,k}(v^{(2)'},I)$, and $\Delta_1(v^{(2)'},I)=\emptyset$. By lemma~\ref{a:lem:relmax}, it means that $v^{(2)'}$ is a relative maximal of $I$. In addition, $v\in\Delta_1(v^{(2)'},\Z^p)$, which contradicts the hypothesis that for all relative maximal $\alpha$, $v\notin\Delta(\alpha,\Z^p)$. Therefore, $v\in\val(I)$. 
\end{proof}

For an irreducible curve, the computation of the set of values of an ideal is described in \cite{hefez-standard}. An algorithm for two branches is suggested in \cite{polvalues}. If the set of relative maximals is known, the previous theorem can be used to compute the set of values of an ideal of a curve with three branches, and by induction for an arbitrary number of branches, provided that at each step the relative maximals are known. 

The question is therefore to compute the relative maximals of $I$. Thanks to theorem~\ref{intro:theo:sym-max}, it is equivalent to the computation of the absolute maximals of $I^\vee$. In the following, we prove that we only need to know the set of irreducible absolute maximals which is defined below and generalizes \cite[Remark 3.11]{delgado2}.

\begin{prop}
Let $R$ be a local admissible ring and let $I\subseteq \mathrm{Frac}(R)$ be a fractional ideal. Let $v\in\val(I)$ be an absolute maximal. Let us assume that $v=\alpha+\beta$ with $\alpha\in\val(R)$ and $\beta\in\val(I)$. Then  $\alpha$ is an absolute maximal of $R$ and $\beta$ is an absolute maximal of $I$. 
\end{prop}
\begin{proof}
Let us assume that $\beta$ is not an absolute maximal of $I$. Then there exists $J\subseteq \lra{1,\ldots,p}$ with $J\neq\emptyset$ and $J\neq\lra{1,\ldots,p}$ such that there exists $w\in\Delta_J(\beta,I)$. Then, $\alpha+w\in\val(I)$ since $\alpha\in\val(R)$ and $w\in\val(I)$. Thus $\alpha+w\in\Delta_J(v,I)$, which contradicts the fact that $v$ is an absolute maximal of $I$. A similar argument can be used to prove that $\alpha$ is an absolute maximal of $R$. 
\end{proof}

\begin{de}
Let $v\in\val(I)$ be an absolute maximal. We call $v$ an \emph{irreducible absolute maximal} if it cannot be written as $v=\alpha+\beta$ with $\alpha\in\val(R)\backslash \lra{0}$ and $\beta\in\val(I)$. More generally, if $v\in\val(I)$, we say that $v$ is \emph{irreducible} if for all $a\in\val(\co_C)$ and for all $b\in\val(I)$, the condition $v=a+b$ implies $a=0$. 
\end{de}

The following proposition is a direct consequence of the definition:

\begin{prop}
Let $G=\lra{g_1,\ldots, g_q}$ be the set of irreducible absolute maximals of $R$ and $\lra{\alpha_1,\ldots,\alpha_r}$ be the set of irreducible absolute maximals of $I$. Then the set of absolute maximals of $I$ is contained in the set $\bigcup_{i=1}^r \lrp{\N g_1+\N g_2+\cdots +\N g_q+ \alpha_i}$. 

\end{prop}

To apply the generation theorem~\ref{theo:gen} to a fractional ideal $I$, we need the absolute maximals of $I^\vee$. 

\begin{nota}
Let $\lra{g_1,\ldots,g_q}$ be the set of irreducible absolute maximals of $R$, and $\lra{\beta_1,\ldots,\beta_s}$ be the set of irreducible absolute maximals of the ideal $I^\vee$. Let $\nu_{I^\vee}\in\Z^p$ be such that $\nu_{I^\vee}+\N^p\subseteq \val(I^\vee)$. We set $$F=\lra{\sum_{j=1}^q \lambda_jg_j+\beta_i \big\vert i\in\lra{1,\ldots,s}, \forall j, \lambda_j\in\N \text{ and } \sum_{j=1}^q \lambda_j g_j+\beta_i\leqslant\nu_{I^\vee}-\underline{1}}\subseteq \val(I^\vee).$$

Since $\nu_{I^\vee}+\N^p\subseteq \val(I^\vee)$, one can notice that if $\beta$ is an absolute maximal of $I^\vee$, then $\beta\leqslant \nu_{I^\vee}-\undun$.

We also set $F'=\lra{\gamma-u-\underline{1}\vert u\in F}$, where $\gamma$ is the conductor of $R$.
\end{nota}

We therefore have the following result:

\begin{prop}
Let $v\in\Z^p$ be such that for all $J\subseteq \unp$ with $|J|=p-1$, we have $pr_J(v)\in pr_J(\val(I))$. Then:
$$v\in \val(I) \iff \forall w\in F', v\notin \Delta(w,\Z^p).$$
\end{prop}

\begin{proof}
Let us assume that $v\in\val(I)$ and that there exists $w\in F'$ such that $v\in{\Delta}(w,\Z^p)$. For example, we may assume that $v\in\Delta_1(w,\Z^p)$. Since $v\in\val(I)$, by theorem~\ref{symmetry-values}, $\Delta(\gamma-v-\underline{1},I^\vee)=\emptyset$. We have $(\gamma-v-\undun)_1=\gamma_1-w_1-1$ and for all $j\geqslant 2$, $(\gamma-v-\undun)_j<\gamma_j-w_j-1$. In addition, $\gamma-w-\undun\in\val(I^\vee)$ by definition of $F$ and $F'$. Thus $\gamma-w-\undun\in \Delta_1(\gamma-v-\undun,I^\vee)$, which is a contradiction. 

The other implication is a consequence of theorem~\ref{theo:gen}. Indeed, the set of absolute maximals of $I^\vee$ is contained in $F$, and by theorem~\ref{intro:theo:sym-max}, the  set of relative maximals of $I$ is contained in $F'$. 
\end{proof}

\subsection{Irreducible absolute maximals of an ideal}
\label{subsection:irred:abs:max}

The remaining problem would be to determine the absolute maximals of $I^\vee$. 

When $C$ is a plane curve, the set of irreducible absolute maximals of $\co_C$ coincide with the values of maximal contact which are finite, see \cite{delgado2} for more details. 

We identify here a subset of the set of irreducible absolute maximals of a fractional ideal $I$ of a reduced reducible germ of curve $C\subseteq (\C^m,0)$. Let $p$ be the number of irreducible components of $C$.

The set of values of a fractional ideal of an irreducible curve can be deduced from a standard basis. An algorithm computing the standard basis of an ideal is given in \cite{hefez-standard}. 

\begin{de}[\protect{\cite[Definition 2.1]{hefez-standard}}]
\label{de:base:standard} Let $\mc{G}=\lra{g_1,\ldots,g_s}\subseteq \co_C$. 
\begin{itemize}
\item A $\mc{G}$-product is an element of the form $\prod_{i=1}^s g_i^{\alpha_i}$ where for all $i\in\lra{1,\ldots,s}$, $\alpha_i\in\N$. 
\item The set $\mc{G}$ is called a \emph{standard basis} of $\co_C$ if for all $f\in \co_C$, there exists a $\mc{G}$-product $g$ such that $\val(g)=\val(f)$. In other words, $$\val(g_1)\N+\cdots+\val(g_s)\N=\val(\co_C).$$
\item Let $\mc{H}\subseteq I$. The couple $(\mc{H},\mc{G})$ is called a \emph{standard basis} of $I$ if $\mc{G}$ is a standard basis of $\co_C$ and if for all $f\in I$, there exist $h\in \mc{H}$ and a $\mc{G}$-product $g$ such that $\val(f)=\val(g)+\val(h)$.
\item Let $(\mc{H},\mc{G})$ be a standard basis of $I$. We say that $\mc{H}$  is \emph{minimal} if for all $h\neq h'\in\mc{H}$ we have $\val(h')\notin \val(\co_C)+\val(h)$.
\end{itemize}
\end{de}

\begin{remar}
Let $(\mc{H},\mc{G})$ be a standard basis of $I$. If there exists $h\neq h'\in\mc{H}$ such that $\val(h')\in\val(\co_C)+\val(h)$, then there is a $\mc{G}$-product $g$ such that $\val(h')=\val(g)+\val(h)$. Then $(\mc{H}\backslash\lra{h'},\mc{G})$ is a also  a standard basis of $I$. By iterating this process, one can deduce a standard basis $(\mc{H}',\mc{G})$ where $\mc{H}'$ is minimal.  
\end{remar}
\begin{nota}

Let for $i\in \unp$, $\pi_i : \mathrm{Frac}(\co_C)\to \mathrm{Frac}(\co_{C_i})$ be the natural surjection. We set $I_i=\pi_i(I)\subseteq \mathrm{Frac}(\co_{C_i})$.
 
For all $i\in\unp$, let us consider a standard basis $(\mc{H}^i,\mc{G}^i)$ of $I_i$ where $\mc{H}^i$ is minimal, and let us write $\mc{H}^i=\lra{H_0^i,\ldots,H_{s_i}^i}$ and $\mc{G}^i=\lra{g_0^i,\ldots,g_{r_i}^i}$. For all $j\in\lra{0,\ldots,s_i}$, let $h^i_{j,i}=\val_i(H^i_j)\in\Z$.
\end{nota}

Let us fix $\nu\in\val(I)$ such that $\nu+\N^p\subseteq \val(I)$. We do not assume that $\nu$ is the conductor of $I$. 

\begin{nota}
For all $i\in\unp$ and for all $j\in\lra{0,\ldots,s_i}$, we set $$E_j^i(\nu)=\lra{v\in\val(I)\vert v_i=h_{j,i}^i, \forall \ell\neq i, v_\ell\leqslant \nu_\ell}.$$
\end{nota}

\begin{lem}
\label{lem:irred}
Let $\alpha\in E_j^i(\nu)$. Then $\alpha$ is an irreducible element. 
\end{lem}
\begin{proof}
Let us assume that $\alpha=a+h$ with $a\in\val(\co_C)$ and $h\in\val(I)$. In particular, $h_{j,i}^i=a_i+h_i$ with $a_i\in\val_i(\co_{C_i})$ and $h_i\in\val_i(I_i)$. Since $\mc{H}^i$ is minimal we have $a_i=0$ and $h_i=h_{j,i}^i$. Therefore, since $\co_C$ is a local ring, one has $a=0$ (see \cite[(1.1.1)]{delgado}). Therefore, $\alpha$ is irreducible. 
\end{proof}

\begin{prop}
\label{prop:max}
Let $\alpha\in E_j^i(\nu)$. Then $\alpha$ is an irreducible absolute maximal of $I$ if and only if for all $\ell\in\unp, \alpha_\ell\neq \nu_\ell$ and $\alpha$ is a maximal element of $E_j^i(\nu)$ for the product order of $\Z^p$.  
\end{prop}
\begin{proof}
Let $\alpha\in E_j^i(\nu)$. If $\alpha$ is not a maximal element of $E_j^i(\nu)$ for the product order, then there exists $\alpha'\in E_j^i(\nu)$ such that $\alpha'\geqslant\alpha$, $\alpha'\neq \alpha$ and $\alpha'_i=\alpha_i$, so that $\alpha'\in\Delta_J(\alpha,I)$ for a subset $J\subseteq\unp$ with $J\neq \emptyset$ and $J\neq\unp$. Thus, $\alpha$ is not an absolute maximal of $I$.

Let us assume that there exists $\ell\in\unp$ such that $\alpha_\ell=\nu_\ell$. We define $\alpha'\in\Z^p$ such that for all $q\neq \ell$, $\alpha'_q=\alpha_q$ and $\alpha'_\ell=\alpha_\ell+1$. Then, since $t^\nu\co_{\wt{C}}\subseteq I$, $\alpha'\in\val(I)$ so that $\alpha'\in\Delta_{\unp\backslash\lra{\ell}}(\alpha,I)$ and $\alpha$ is not an absolute maximal of $I$. 

Let us prove the converse implication. Let $\alpha\in E_j^i(\nu)$ be such that for all $\ell\in\unp, \alpha_\ell\neq \nu_\ell$ and $\alpha$ is a maximal element of $E_j^i(\nu)$. Let us prove that $\alpha$ is an irreducible absolute maximal of $I$. By lemma~\ref{lem:irred}, $\alpha$ is irreducible. Let us assume that $\alpha$ is not an absolute maximal of $I$. There exists $J\subseteq \unp$, $J\neq\emptyset$ and $J\neq \unp$ such that $\Delta_J(\alpha,I)\neq \emptyset$. Let $w\in\Delta_J(\alpha,I)$.

If $i\in J$, then by proposition~\ref{a:prop:inf}, $w':=\inf(\nu,w)\in E_j^i(\nu)$, $w'\geqslant \alpha$, and $w'\neq \alpha$, which contradicts the maximality of $\alpha$. 

Thus, $i\notin J$. We then have $w_i>\alpha_i$, and there exists $\ell\in\unp$, $\ell\neq i$, such that $\alpha_\ell=w_\ell$. Let us apply proposition~\ref{a:prop:valquimonte} to the triple $(\alpha,w,\ell)$. There exists $v\in\val(I)$ such that $v_i=\alpha_i$, $v_\ell>\alpha_\ell$, and for all $q\notin \lra{i,\ell}$, $v_q\geqslant \min(\alpha_q,w_q)\geqslant\alpha_q$. Therefore, by proposition~\ref{a:prop:inf}, $v':=\inf(v,\nu)\in E_j^i(\nu)$, $v'\geqslant\alpha$ and $v'_\ell>\alpha_\ell$, which contradicts the fact that $\alpha$ is a maximal element in $E_j^i(\nu)$. Hence the result.
\end{proof}

\begin{lem}
The set of the elements $\alpha\in E_j^i(\nu)$ such that for all $\ell\in\unp, \alpha_\ell\neq \nu_\ell$ and $\alpha$ is a maximal element of $E_j^i(\nu)$ does not depend on the choice of an element $\nu$ satisfying $\nu+\N^p\subseteq \val(I)$.
\end{lem}
\begin{proof}
Let $\nu, \nu'$ be such that $\nu+\N^p\subseteq \val(I)$ and $\nu'+\N^p\subseteq \val(I^\vee)$. 
Let $\alpha\in E_j^i(\nu)$ be such that for all $\ell\in\unp, \alpha_\ell\neq \nu_\ell$ and $\alpha$ is a maximal element of $E_j^i(\nu)$. Let us prove that $\alpha\in E_j^i(\nu')$. Let us set $\alpha'=\inf(\alpha, \nu')$. Let us prove that for all $\ell\in\unp, \alpha'_\ell\neq\nu'_\ell$. Let us assume that there exists $L\subseteq \unp$, $L\neq \emptyset$, such that for all $\ell\in L$, $\alpha'_\ell=\nu'_\ell$. In particular, it means that for all $\ell\in L$, $\nu'_\ell<\nu_\ell$. Let $f\in I$ be such that $\val(f)=\alpha'$. Since $t^{\nu'}\co_{\wt{C}}\subseteq I$, one can choose an element $g\in I$ such that the restriction $g|_{C_q}$ of $g$ to $C_q$ for $q\notin L$ is zero, and for all $\ell\in L$, $g|_{C_\ell}=f|_{C_\ell}+t_\ell^{\nu_\ell}$, so that for $q\notin L$, $\val_q(f-g)=\alpha_q$ and for $\ell\in L$ $\val_\ell(f-g)=\nu_\ell$. Then $\val(f-g)\in E_j^i(\nu)$ and $\alpha\leqslant \val(f-g)$, which contradicts the maximality of $\alpha$ in $E_j^i(\nu)$. Therefore, for all $q\in\unp$, $\alpha_q<\nu_q'$. In particular, $\alpha\in E_j^i(\nu')$. 

It remains to prove that $\alpha$ is a maximal element of $E_j^i(\nu')$. Let us assume that $\alpha$ is not a maximal element in $E_j^i(\nu')$. Let $v\in E_j^i(\nu')$ be such that $v\geqslant \alpha$ and $v\neq \alpha$. Let $v'=\inf(v,\nu)$. Then $v'\in E_j^i(\nu)$, $v'\geqslant \alpha$ and since for all $j, \alpha_j\neq \nu_j$, $v'\neq \alpha$, it contradicts the maximality of $\alpha$ in $E_j^i(\nu)$.
\end{proof}

\begin{remar}
We identify in proposition~\ref{prop:max} the irreducible absolute maximals which have at least one coordinate equal to the valuation of an element in a standard basis of the ideal along the corresponding branch. It is not obvious that this property is satisfied for all irreducible absolute maximals of $I$. The computation of the irreducible absolute maximals in \cite{delgado2} relies on the Hamburger-Noether expansion of plane curves and elements of maximal contact which are  particular to the semigroup of a plane curve. 
\end{remar}

\bibliographystyle{alpha-fr}
\bibliography{bibli2}
\end{document}